\documentclass[11pt,leqno]{amsart}

\usepackage{hyperref}
\usepackage{amsmath,amsfonts,amssymb}
\usepackage{enumerate}
\usepackage{xcolor}
\usepackage{url}
\usepackage{amsmath}
\usepackage[arrow, matrix, curve]{xy}
\usepackage{amssymb}
\usepackage{a4wide}
\usepackage{tikz}
\usepackage{amsthm}
\usepackage{latexsym}
\theoremstyle{plain}
\newtheorem{thm}{Theorem}[section]
\newtheorem{lemma}[thm]{Lemma}
\newtheorem{cor}[thm]{Corollary}
\newtheorem{prop}[thm]{Proposition}
\theoremstyle{definition}

\theoremstyle{remark}

\numberwithin{equation}{section}

\newcommand{\CC}{\mathbb C}

\newcommand{\PP}{\mathbb P}

\newcommand{\ZZ}{\mathbb Z}

\begin{document}
\title{Varieties isogenous to a higher product with prescribed numerical invariants}
\author{Amir D\v{z}ambi\'c}
 \address{Amir D\u{z}ambi\'{c}: 
Mathematisches Seminar,
Christian-Albrechts-Universit\"{a}t zu Kiel,
24118 Kiel, Germany}
 \email{dzambic@math.uni-kiel.de}

 \author{Anitha Thillaisundaram}
 
 \address{Anitha Thillaisundaram: 
 Centre for Mathematical Sciences, Lund University,  223 62 Lund, Sweden}
 \email{anitha.thillaisundaram@math.lu.se}

 \thanks{The first author acknowledges support from the Deutsche Forschungsgemeinschaft (DFG), grant DZ 105/02. The second author acknowledges support from the Folke Lann\'{e}r's Fund and she thanks the Christian-Albrecht University of Kiel for its hospitality. This research was also supported by the Royal Physiographic Society of Lund.} 

 \date{\today}
\maketitle
\begin{abstract} 
Using structural properties of groups of small order, we establish the non-existence of varieties isogenous to a higher product of dimension $n$ greater than 3 with  fixed topological Euler number $(-2)^n$ and trivial first Betti number.
\end{abstract}

\section{Introduction}
A smooth complex projective variety $X$ of dimension $n$ is said to be \emph{isogenous to a higher product}, if there exists a finite unramified covering $\pi: C_1\times\cdots\times C_n\to X$ where $C_1,\ldots ,C_n$ are compact Riemann surfaces, all of genus greater
than $1$. Equivalently, $X$ is isogenous to a higher product if  $X=(C_1\times\cdots\times C_n)/G$ where $G$ is a finite group that acts freely on $C_1\times\cdots\times C_n$ as a group of holomorphic automorphisms. Varieties isogenous to a higher product have been introduced by Catanese in \cite{Catanese00} and have proven to be very useful in several geometrical contexts. For instance, following an idea of Beauville, they have been used to give examples of algebraic surfaces~$S$ of general type with $p_g(S)=q(S)=0$ and $K_S^2=8$; here $K_S$, $p_g(S)=\dim H^{0}(S,\mathcal O(K_S))$ and $q(S)=\dim H^1(S,\mathcal O_S)$ are respectively the canonical divisor, the geometric genus (i.e. the dimension of the space of global sections of the sheaf corresponding to $K_S$) and the irregularity (i.e. the dimension of the first cohomology group of the structure sheaf $\mathcal O_S$ of $S$). Recall that an algebraic variety $X$ is said to be of \emph{general type}, if its Kodaira dimension, i.e. the transcendence degree of the canonical ring  minus one, equals its dimension. Note that a product of curves of genus at least two is of general type. Since the Kodaira dimension is invariant under unramified morphisms, every variety isogenous to a higher product is of general type. We remark that the condition  $p_g(S)=q(S)=0$ for a surface is equivalent to $b_1(S)=0$, where $b_1$ denotes the first Betti number. As another application, surfaces isogenous to a higher product have been used to provide counterexamples to the so-called $\mathbf{diff}=\mathbf{def}$  speculation of Friedman and Morgan asking if for minimal surfaces of general type the diffeomorphism classes coincide with their deformation types; see \cite{Catanese03}. Lastly we mention that surfaces isogenous to a higher product can be used to find examples of surfaces of general type with the canonical map of high degree; see \cite{Catanese18} and \cite{Gleissneretal22}.

Let $G^0=G\cap (\text{Aut}(C_1)\times\cdots\times \text{Aut}(C_n))$, where $\text{Aut}(C_i)$ denotes the group of holomorphic automorphisms of $C_i$. We say that $X$ is of \emph{unmixed type}, if $G=G^{0}$, i.e. if $G$ acts diagonally on $C_1\times\cdots\times C_n$ and of \emph{mixed type} otherwise. The action of $G$ on the product may in general not be faithful, but even if the action is faithful, the action on the factors induced by projection may not be faithful.  We say that $G$ acts \emph{absolutely faithfully} on $C_1\times\cdots\times C_n$ if the action on $C_i$ induced from the projection $C_1\times\cdots\times C_n\to C_i$ is faithful for each $i=1,\ldots,n$.

Here we are interested in the question of the existence of varieties $X$ isogenous to a higher product with some fixed numerical invariants. Specifically we ask for the existence of $X$ with topological Euler number $e(X)=(-2)^n$ and vanishing first Betti number.
Let us define: 
\begin{align*}
&\mathcal S_0(e,n)=\\
&\,\,\{ X=(C_1\times\cdots\times C_n)/G\mid e(X)=e, \ b_1(X)=0,\ X \text{ of unmixed type},\ G \text{ absolutely faithful}\}
\end{align*}

Our main result is the following:
\begin{thm}\label{thm:main}
    For all $n\ge 4$, the set $\mathcal{S}_0((-2)^n,n)$ is empty.
 \end{thm}
The result above is in contrast to the situation for $n<4$. Indeed, examples of corresponding varieties 
in dimensions $n=2$ and $n=3$ are known from the work of Bauer, Catanese and Grunewald \cite{BauerCataneseGrunewald08} (dimension $2$) and from Frapporti and Glei{\ss}ner \cite{FrapportiGleissner16} (dimension 3). Moreover by \cite{BauerCataneseGrunewald08} and \cite{FrapportiGleissner16} the sets $\mathcal S_0(4,2)$ and $\mathcal S_0(-8,3)$ can be described explicitly.

We prove our main result by first establishing the following: 
\begin{thm}
	\label{mainthm}
Let $X=(C_1\times\cdots\times C_n)/G\in \mathcal S_0((-2)^n,n)$. Then all of the factors $C_i$ have genus $g(C_i)>2$. Furthermore, if $\mathcal S_0((-2)^n,n)\neq \varnothing$, then $n\leq 5$. 
\end{thm} 
\noindent Using Theorem~\ref{mainthm}, we then show that there is no finite group $G$ that satisfies all the necessary conditions for $X=(C_1\times\cdots\times C_n)/G$ to be in $\mathcal{S}_0((-2)^n,n)$.

Our initial motivation for considering varieties isogenous to a higher product was the search for varieties of general type with the same Betti numbers as the $n$-fold product $\mathbb P_1(\mathbb C)^n$ of  projective lines. We call such varieties \emph{fake $\mathbb P_1(\mathbb C)^n$}. A natural source for examples of fake $\mathbb P_1(\CC)^n$ come from varieties that are analytically isomorphic to the quotients $\mathbb H^n/\Gamma$ of the $n$-fold product of upper half planes by cocompact lattices $\Gamma$ inside $\text{PSL}_2(\mathbb R)^n$. The motivation comes from  Hirzebruch's proportionality principle which implies that the numerical invariants of such quotients are closely related to the invariants of $\mathbb P_1(\mathbb C)^n$. And indeed, several examples of fake $\mathbb P_1(\CC)^n$ of this type have been found where $\Gamma$ is an irreducible lattice; see \cite{dza1}, \cite{dza2}, \cite{Shavel}. Recall that a lattice $\Gamma\subset \text{PSL}_2(\mathbb R)^n$ is called \emph{irreducible} if its image under any non-trivial projection $\text{PSL}_2(\mathbb R)^n\to \text{PSL}_2(\mathbb R)^m$, for $m<n$, is a dense subgroup. Somehow complementary to this is the situation where $\Gamma$ is a lattice that contains a product $\Pi_1\times \cdots \times \Pi_n$ of surface groups $\Pi_k$, for $k=1,\ldots,n$, of genus at least $2$. Such a lattice may be called \emph{totally reducible} and in this case the quotient $\mathbb H^n/\Gamma$ is a variety isogenous to a higher product. Since fake $\PP_1(\CC)^n$ have first Betti number equal to zero,  Theorem~\ref{thm:main} implies:
       
\begin{cor}
    For $n\ge 4$, there are no fake $\mathbb{P}_1(\mathbb{C})^n$ isogenous to a higher product.
\end{cor}
Note that the aforementioned surfaces isogenous to  a higher product discovered by Bauer, Catanese and Grunewald indeed give examples of fake $\mathbb P_1 (\mathbb C)^2$. The examples by Frapporti and Glei{\ss}ner are however not relevant, since the topological Euler number of a variety isogenous to a product of odd dimension is negative whereas the Euler number of the product of projective lines is always positive. Therefore the varieties isogenous to a higher product in odd dimensions cannot be fake $\mathbb{P}_1(\mathbb{C})^n$.

Finally we remark that in this paper we do not consider the most general situation, having restricted to the unmixed, absolutely faithful and free actions. Our aim is to study the situations with weaker conditions on $G$ in future, with the hope of finding examples of varieties isogenous to a higher product with  interesting geometric properties. In small dimensions this has already been considered; see \cite{BCGP12}. A theory of resolution of cyclic quotient singularities in arbitrary dimensions has been developed in ~\cite{Fujiki}. We note that the discussion of non-free actions of finite groups on products of curves is also interesting from the point of view of fundamental groups; see \cite{BCGP12} and \cite{DP}.\\

\smallskip

\noindent\textit{Organisation.}  In Section~\ref{sec:2} we collect together some useful results concerning invariants of varieties isogenous to higher products. Then in Section~\ref{sec:3}, we establish conditions that the group $G$ has to satisfy in order for $G$ to define a variety isogenous to a higher product with our prescribed numerical invariants. In Section~\ref{sec:4} we prove Theorem~\ref{mainthm}. Finally, in Section~\ref{sec:5}, using the results of the previous section together with combinatorial arguments, we first narrow down the possible options for the order of $G$, 
to then  complete the proof of Theorem~\ref{thm:main}.

\medskip

\noindent\textbf{Acknowledgements.} We thank Barbara Langfeld for her interest in the subject and her help by confirming several of our results using a computer. We also thank the referee for pointing out the work of Polizzi~\cite{Pol09}, which closes a gap in a result of Broughton~\cite[Table~5]{Broughton91}.

\section{Invariants of varieties isogenous to higher products} \label{sec:2}
The theory of varieties isogenous to a higher product was developed by Catanese in \cite{Catanese00}. We refer to this fundamental paper for all details on these varieties. 

As indicated in the introduction our standing assumption for the rest of the paper is that  $X=(C_1\times\cdots\times C_n)/G$ is a variety of unmixed type and where $G$ acts faithfully on each of the factors $C_i$. Recall that  $e(X)$ denotes the topological Euler number of $X$, and that $g(C_i)>1$ denotes the genus of $C_i$.

The following result is fundamental for our purposes and is well known. For the reader's convenience, we include a proof.
  
\begin{lemma}\label{eulerzahl}
Let $X=(C_1\times\cdots\times C_n)/G$ be a variety isogenous to a higher product. Then
\[
e(X)=\frac{(-2)^n}{|G|}\prod_{i=1}^n (g(C_i)-1).
\]
\end{lemma}

\begin{proof}
For a compact Riemann surface $C$ of genus $g$ the Euler number equals $e(C)=2-2g=-2(g-1)$. As the Euler number is multiplicative with respect to taking products we have $e(C_1\times\cdots\times C_n)=(-2)^n\prod_{i=1}^n(g(C_i)-1)$. Finally, since $\pi:C_1\times\cdots\times C_n\rightarrow X$ is a finite unramified covering, we have $e(C_1\times\cdots\times C_n)=\deg(\pi)e(X)=|G|e(X)$.
\end{proof}

As an immediate consequence we obtain:

\begin{cor}
	\label{orderG}
Let $X=(C_1\times\cdots\times C_n)/G$ be a variety isogenous to a higher product with $e(X)=(-2)^n$. Then $|G|=\prod_{i=1}^n(g(C_i)-1)$.
\end{cor}

We now analyse the condition under which a variety isogenous to a higher product has vanishing first Betti number.

\begin{lemma}
	\label{kunneth}
Let $X=(C_1\times\cdots\times C_n)/G$ be a variety isogenous to a higher product. Then $b_1(X)=0$ if and only if for all $i\in\{1,\ldots,n\}$ we have $C_i/G\cong \PP_1(\CC)$.
\end{lemma}

\begin{proof}
We will make use of K\"{u}nneth's formula. Namely, the first cohomology with complex coefficients of the product $C_1\times\cdots\times C_n$, is isomorphic to
\[
H^1(C_1\times\cdots\times C_n,\CC)\cong \bigoplus_{k_1+\cdots+k_n=1} H^{k_1}(C_1,\CC)\otimes\cdots\otimes H^{k_n}(C_n,\CC).
\]
Here, the sum runs over all $n$-tuples $(k_1,\ldots,k_n)$ of non-negative integers whose sum equals to $1$, i.e. exactly one of the entries equals $1$ and all the others equal $0$. Since $H^0(C_i,\CC)\cong \CC$ for all $i$ we obtain 
\[
H^1(C_1\times\cdots \times C_n,\CC)\cong \bigoplus_{i=1}^n H^1(C_i,\CC).
\]  
Now we use the fact that $H^1((C_1\times\cdots\times C_n)/G,\CC)\cong H^1(C_1\times\cdots\times C_n,\CC)^G$, where the latter space is the subspace of $H^1(C_1\times\cdots\times C_n,\CC)$ on which the (linear) action of $G$ lifted to cohomology is trivial. By the linearity of the lifted action we have 
\begin{align*}
H^1(X,\CC)&\cong H^1(C_1\times\cdots\times C_n,\CC)^G \\
&\cong \left( \bigoplus_{i=1}^n H^1(C_i,\CC)\right)^G\\
&\cong  \bigoplus_{i=1}^n H^1(C_i,\CC)^G \\
&\cong \bigoplus_{i=1}^n H^1(C_i/G,\CC).
\end{align*} 
Obviously $H^1(X,\CC)$ equals zero if and only if all the summands $H^1(C_i/G,\CC)$ are zero. This means that the quotient $C_i/G$ has genus zero, i.e. $C_i/G\cong \PP_1(\CC)$. 
\end{proof}

\section{Possible ramification structures of genus $0$ on finite groups} \label{sec:3}

In this section we will discuss properties of a finite group $G$ that defines a variety $X=(C_1\times\cdots\times C_n)/G$ isogenous to a higher product (acting absolutely faithfully and of unmixed type) such that $C_i/G\cong \PP_1(\CC)$ for every $i\in\{1,\ldots,n\}$ and with $b_1(X)=0$. Certainly we will always assume that $G$ is non-trivial.

We observe that the condition $C_i/G\cong\PP_1(\CC)$ implies that $G$ is a homomorphic image of a cocompact Fuchsian group $\Lambda$ of genus zero. This implies that we can find $n$ systems $T_i=T_i(G):=(h^{(i)}_1,\ldots,h^{(i)}_{r(i)})\in G^{r(i)}$, where $r(i)\in\mathbb{N}$ for $i\in\{1,\ldots,n\}$, of generators of $G$ such that $h^{(i)}_1\cdots h^{(i)}_{r(i)}=1$. Such generating systems are called \emph{spherical systems of generators} of~$G$. To a spherical system $T=(h_1,\ldots,h_r)$ of $G$, we associate a tuple $A=[m_1,\ldots,m_{r}]\in \mathbb N^{r}$, where $m_j=\text{o}(h_{j})$ is the order of $h_{j}$. We call $A$ the \emph{ramification type} of $T$.  

First we recall the group-theoretic condition for a finite group $G$, acting on Riemann surfaces $C_1,\ldots, C_n$ each of genus at least 2, where the action on $C_i$ is determined by the system  $T_i=(h_1^{(i)},\ldots,h_{r(i)}^{(i)})$ of spherical generators, to act freely on $C_1\times \cdots\times C_n$. For $i\in\{1,\ldots,n\}$, let $A_i=[m_1^{(i)},\ldots,m_{r(i)}^{(i)}]$ be the ramification type associated with $T_i$. We define for each $i\in\{1,\ldots,n\}$,
\[
\Sigma(T_i)=\bigcup_{j=1}^{r(i)}\bigcup_{k\in\mathbb N}\bigcup_{g\in G}\{g(h_j^{(i)})^{k}g^{-1}\}.
\]

\begin{lemma}
\label{sigma}
 Let $G$ be a finite group acting on Riemann surfaces $C_1,\ldots, C_n$ each of genus at least 2, and let $T_1,\ldots, T_n$ be the respective spherical systems of generators. Then, the action of $G$ on $C_1\times\cdots\times C_n$ is free if and only if 
\begin{equation}
	\label{sigmacondition}
\bigcap_{i=1}^{n}\Sigma(T_i)=\{1_G\}.
\end{equation}
\end{lemma}

\begin{proof}
An element $g\in G$ will have a fixed point on $C_i$ if and only if $g$ is a conjugate of some power $(h_j^{(i)})^k$ of a spherical generator from $T_i$. By definition, $\Sigma(T_i)$ is then the set of elements of $G$ that have a fixed point on $C_i$. Therefore, an element $g\in G$ will have a fixed point on $C_1\times\cdots\times C_n$ if and only if $g\in \bigcap_{i=1}^{n}\Sigma(T_i)$. 
\end{proof}

For a given spherical system $T_i=(h^{(i)}_1,\ldots,h^{(i)}_{r(i)})$ of generators of $G$ and associated ramification type $A_i=[m_1^{(i)},\ldots, m_{r(i)}^{(i)}]$ we define 
\[
\Theta(A_i)=-2+\sum_{j=1}^{r(i)}\left( 1-\frac{1}{m_j^{(i)}}\right)\qquad \text{and}\qquad \alpha(A_i)=\frac{2}{\Theta(A_i)}.
\]  
In the following we will assume that entries of a type $A=[m_1,\ldots,m_r]$ corresponding to a system of spherical generators $T=(h_1,\ldots,h_r)$ satisfy $m_1\leq m_2\leq\ldots\leq m_r$. This is mainly for convenience and will cause no loss of generality. Note that this situation cannot be established by simply re-enumerating the generators, since then, in general the condition $h_1\cdots h_r=1$ will not be satisfied.  

\begin{lemma}
	\label{Theta+alpha}
Let $X=(C_1\times\cdots\times C_n)/G\in \mathcal S_0((-2)^n,n)$.  Then we have for each $i\in\{1,\ldots,n\}$:
\begin{enumerate}[(i)]
\item[\textup{(i)}] $2(g(C_i)-1)=|G|\Theta(A_i)$;
\item[\textup{(ii)}] $\Theta(A_i)=2(g(C_i)-1)/|G|>0$;
\item[\textup{(iii)}] $\alpha(A_i)=2/\Theta(A_i)=\prod_{j\neq i} (g(C_j)-1)\in\mathbb N$.
\end{enumerate}
\end{lemma}
Note that part (i) is referred to as the Riemann-Hurwitz relation.
\begin{proof}
The statement (i) is simply the Riemann-Hurwitz formula. Part (ii) is obvious from (i) (as $g(C_i)>1$). The last statement is clear from (ii) and Corollary~\ref{orderG}.
\end{proof}

The following lemma restricts the possible orders of the spherical generators of $G$ that define $X$. Its proof is based on arguments from \cite[Proposition 2.3]{BauerCataneseGrunewald08} in the case $n=2$.

\begin{lemma}
	\label{upper_bound_on_m}
Let $X=(C_1\times\cdots\times C_n)/G\in \mathcal S_0((-2)^n,n)$. For each fixed $i\in\{1,\ldots,n\}$ we have $m_j^{(i)}\mid \alpha(A_i)$ for all $j\in\{1,\ldots,r(i)\}$.
\end{lemma}

\begin{proof}
To simplify the notation we omit the superscript $i$ associated to the factor $C_i$, which is fixed in this proof, when considering elements $h_j=h_j^{(i)}$ of $T_i$. Accordingly we write $m_j=m_j^{(i)}$ and $r=r(i)$. Let $j\in\{1,\ldots, r\}$ be arbitrary. As $G$ is embedded in $\text{Aut}(C_i)$ the element  $h_j$ of order  $m_j$ has a fixed point on $C_i$, but there exists $k\neq i$ such that $h_j$ acts without fixed points on $C_k$. Otherwise $h_j$ would have a fixed point on $C_1\times\cdots\times C_n$ and this would contradict the assumption that $G$ acts freely on this product. We consider the covering 
\[
C_k\longrightarrow C_k/\langle  h_j\rangle=:\widetilde{C}_k.
\]     
By the choice of $C_k$ this covering is unramified and by the Riemann-Hurwitz relation we have
\[
 m_j\left(g(\widetilde{C}_k)-1\right)=g(C_k)-1.
\]
Since $g(C_k)>1$ we have $g(\widetilde{C}_k)>1$. It follows immediately that $m_j\mid (g(C_k)-1)$. Since $k\neq i$, Lemma \ref{Theta+alpha}(iii) implies that $m_j\mid \alpha(A_i)$.
\end{proof}

In light of the above result, we recall the following result from \cite[Proposition 1.4]{BauerCataneseGrunewald08}:
\begin{cor}
	\label{Listoforders}
Let $X=(C_1\times\cdots\times C_n)/G\in \mathcal S_0((-2)^n,n)$. For each $i\in\{1,\ldots,n\}$ we have $3\leq r(i)\leq 6$. In fact we have the following further restrictions:
\begin{enumerate}
\item[\textup{(a)}] $r(i)=6$ if and only if $A_i=[2,2,2,2,2,2]$ with corresponding $\alpha(A_i)=2$;\\
\item[\textup{(b)}] $r(i)=5$ if and only if $A_i=[2,2,2,2,2]$ with $\alpha(A_i)=4$;\\
\item[\textup{(c)}] $r(i)=4$  if and only if $A_i$ and its corresponding $\alpha(A_i)$ belong to the following list: 

\begin{flushleft}
	\begin{tabular}{|c||c|c|c|c|c|c|}
		\hline
		$A_i$ & $[2,2,2,3]$ & $[2,2,2,4]$ & $[2,2,2,6]$ & $[2,2,3,3]$ & $[2,2,4,4]$ &         $[3,3,3,3]$ \\
		\hline
		$\alpha(A_i)$ & $12$ & $8$ & $6$ & $6$ & $4$         & $3$ \\ 
		\hline
	\end{tabular}\\
\end{flushleft}

\medskip

\item[\textup{(d)}] $r(i)=3$ if and only if $A_i$ belongs to one of the following $22$ triples, with the corresponding $\alpha(A_i)$ value:

\begin{flushleft}
	\begin{tabular}{|c||c|c|c|c|c|c|c|c|}
	\hline
	$A_i$ & $[2,3,7]$ & $[2,3,8]$ & $[2,4,5]$ & $[2,3,9]$ & $[2,3,10]$ & $[2,3,12]$ & $[2,4,6]$ & $[3,3,4]$  \\
	\hline
	$\alpha(A_i)$ & $84$ & $48$ & $40$ & $36$ & $30$ & $24$ & $24$ & $24$ \\ 
	\hline
	\end{tabular}
 \end{flushleft}

 \smallskip

\begin{flushleft}
 \begin{tabular}{|c||c|c|c|c|c|c|c|c|}
  \hline
	$A_i$     & $[2,5,5]$ & $[2,3,18]$ & $[2,4,8]$     & $[3,3,5]$ & $[2,4,12]$ & $[2,6,6]$ & $[3,3,6]$ & $[3,4,4]$\\
	\hline
	$\alpha(A_i)$     & $20$ & $18$ & $16$     & $15$ & $12$ & $12$ & $12$ & $12$ \\ 
 \hline
	\end{tabular}
  \end{flushleft}

 \smallskip
 \begin{flushleft}
 \begin{tabular}{|c||c|c|c|c|c|c|}
  \hline
	$A_i$  & $[2,5,10]$  & $[3,3,9]$ & $[2,8,8]$   & $[4,4,4]$ & $[3,6,6]$    & $[5,5,5]$ \\
	\hline
	$\alpha(A_i)$ & $10$ &  $9$ & $8$     & $8$ & $6$    & $5$ \\ 
 \hline
	\end{tabular}
  \end{flushleft}
\end{enumerate}
\end{cor}

\smallskip

We end this section with the following straightforward application of the previous result.
\begin{cor}
	\label{cor:alpha-1}
Let $X=(C_1\times\cdots\times C_n)/G\in \mathcal S_0((-2)^n,n)$. Then  we have at least two distinct indices $i_1,i_2\in\{1,\ldots,n\}$ such that $g(C_{i_1})$ and $g(C_{i_2})$ are both strictly greater than 2.
\end{cor}

\begin{proof}
    Suppose on the contrary that there is at most one such index. By Corollary~\ref{orderG}, we in fact may assume that there is exactly one such index~$i$. Again using Corollary~\ref{orderG}, it then follows that $G$  acts on a surface of genus $|G|+1$ and with quotient isomorphic to $\PP_1(\mathbb{C})$. Hence there exists a spherical generating system $T=(h_1,\ldots,h_r)$ of $G$ with $A=[\text{o}(h_1),\ldots,\text{o}(h_r)]=[m_1,\ldots,m_r]$ satisfying the Riemann-Hurwitz relation  $2=\Theta(A)$ (see Lemma~\ref{Theta+alpha}(i)), or equivalently $\alpha(A)=1$. But the list of possible ramification types from Corollary~\ref{Listoforders} does not contain any such type. This yields the desired contradiction.
\end{proof}

\section{Dimension bounds} \label{sec:4}

We begin with the observation that $G$ cannot be cyclic, even in a more general setting, as noted below.

\begin{lemma}
	\label{notcyclic}
Let $n\geq 2$ and $X=(C_1\times\cdots\times C_n)/G\in \mathcal S_0(e,n)$ with $e$ arbitrary. Then $G$ is not cyclic.
\end{lemma}

\begin{proof}
Suppose for a contradiction that $G=\langle g\rangle$ is cyclic of order $d$. By the assumption we have $n$ different systems $$T_1=(h^{(1)}_1,\ldots,h^{(1)}_{r(1)}),\,\, \ldots\,\,,\,\,T_n=(h^{(n)}_1,\ldots,h^{(n)}_{r(n)})
$$ of spherical generators of $G$ with respective orders $$
A_1=[m_1^{(1)},\ldots,m_{r(1)}^{(1)}],\,\, \ldots\,\,,\,\,A_n=[m_1^{(n)},\ldots,m_{r(n)}^{(n)}].
$$
For each $i\in\{1,\ldots,n\}$, since $T_i$ generates a cyclic group of order $d$ we have that $d=\text{lcm}(m_{1}^{(i)},\ldots,m_{r(i)}^{(i)})$. On the other hand the condition (\ref{sigmacondition}) implies $\text{gcd}(m_{j_1}^{(1)},\ldots,m_{j_n}^{(n)})=1$ for every choice of $j_1,\ldots j_n$. Indeed, otherwise there exists $t>1$ such that some powers of $h^{(1)}_{j_1}, \ldots, h^{(n)}_{j_n}$ will all have order $t$ and therefore generate the same subgroup of $G$. Then, there will be non-trivial powers of $h^{(k)}_{j_k}$ that coincide. This contradicts (\ref{sigmacondition}). Let $p$ be a prime divisor of $d$. Then for each $i\in\{1,\ldots,n\}$, from $\text{lcm}(m_{1}^{(i)},\ldots,m_{r(i)}^{(i)})=d$  we have $p\mid m_{k_i}^{(i)}$ for some $k_i$. But then $p\mid \text{gcd}(m_{k_1}^{(1)},\ldots,m_{k_n}^{(n)})$, a contradiction.
\end{proof}

Next we find that the genus of every factor $C_i$ is greater than $2$. 

\begin{prop}
	\label{genus>2}
Let $n\geq 2$, and suppose $X=(C_1\times\cdots\times C_n)/G\in \mathcal S_0((-2)^n,n)$.  Then, for each $C_i$ the genus satisfies $g(C_i)>2$.
\end{prop}

\begin{proof}
From Corollary \ref{cor:alpha-1} we know that not all $C_i$ have genus $g(C_i)=2$. Now we assume that one of the $C_i$, say $C_1$, is of genus $2$. As $G$ is absolutely faithful it embeds into $\text{Aut}(C_1)$. The automorphism groups of surfaces of genus $2$ are all classified by the work of Broughton~\cite{Broughton91}. From Lemma~\ref{notcyclic} and the classification  result \cite[Table 4]{Broughton91} we then know that $G$ is one of the following groups:
\[
Z_2\times Z_2, \,D_3, \,Q_8,\, D_4, \,Z_2\times Z_6, \,D_{4,3,-1}, \,D_6,\,D_{2,8,3}, \,Z_3\rtimes D_4, \,\text{SL}_2(3),\, \text{GL}_2(3),
\]   
where we denote the cyclic group of order $d$ by $Z_d$, the dihedral group of order $2m$ by $D_m$, the quaternion group of order $8$ by $Q_8$, and  $D_{p,q,r}=\langle x,y\mid x^p=y^q=1, xyx^{-1}=y^r\rangle$. We remark that  for $Z_3\rtimes D_4$ the kernel of the semidirect action of $D_4$ on $Z_3$ is isomorphic to $Z_2\times Z_2$.

By our initial observation, we know that $G$ must act on at least one surface $C_n$ of genus greater than $2$ and such that $|G|=\prod_{i=1}^n (g(C_i)-1)\geq g(C_n)-1$. We consider now the possible cases.

\medskip

\noindent\underline{Case (a):} $G=Z_2\times Z_2$. Now
Corollaries~\ref{orderG} and \ref{cor:alpha-1} imply that the vector of genera $(g(C_1),\ldots,g(C_n))$ is of the form $(2,\ldots,2,3,3)$. Referring to  Lemma~\ref{Theta+alpha}(i) and Corollary~\ref{Listoforders}, we deduce that the Klein 4-group~$G$ can act on a genus $2$ surface where the quotient is $\PP_1(\mathbb{C})$ only with ramification type $[2,2,2,2,2]$ and respectively on a genus 3 surface with ramification type $[2,2,2,2,2,2]$. In any case, every possible system of spherical generators must contain all elements of $G$, so that the condition (\ref{sigmacondition}) is never satisfied. 

\medskip

\noindent\underline{Case (b)}: $|G|=12$, i.e. $G\in\{Z_2\times Z_6, D_6,D_{4,3,-1}\}$. From Corollaries~\ref{orderG} and \ref{cor:alpha-1}  the vector of genera $(g(C_1),\ldots,g(C_n))$ is either $(2,\ldots,2,3,7)$, $(2,\ldots,2,3,3,4)$ or $(2,\ldots,2,4,5)$. Let $T$ be a spherical system of generators of $G$ that defines an action of~$G$ on a genus $g$ surface with quotient of genus zero and $A$ the ramification type of  $T$. By  Lemma~\ref{Theta+alpha}(i), we have $(g-1)=6\Theta(A)$, or equivalently $(g-1)=12/\alpha(A)$. 

Assuming $g=7$ we get $\alpha(A)=2$ and Corollary \ref{Listoforders} gives $[2,2,2,2,2,2]$ as the only possibility. But it is clear that $G\ne D_6$ cannot have a system of generators with these orders; indeed, all elements of order~2 are central. For $G=D_6$, we instead consider the action on a surface of genus 2. Here $\alpha(A)=12$ and $A\in\{[2,2,2,3],[2,6,6],[3,3,6]\}$; indeed, the group~$D_6$ has no element of order 4. Writing $D_6=\langle x,y\mid x^6=y^2=1,\,yxy=x^{-1}\rangle$, we readily see that a generating set cannot have type $[3,3,6]$ as the only elements of order divisible by~3 are all powers of $x$. Similarly, in a generating system of type $[2,6,6]$  no combination of the generators can be multiplied together to give the identity element. For type $[2,2,2,3]$, note that the only elements of order 3 are $x^2$ and $x^4$, the only central non-trivial element is $x^3$ which is of order 2, and the remaining elements of order 2 are 
\begin{equation}\label{eq:ord-2}
y,\,xy,\,x^2y,\,x^3y,\,x^4y,\,x^5y.
\end{equation}
If all three order-2 generators in the system of type $[2,2,2,3]$ contain a $y$, then any product of all generators of this system is non-trivial, since it is a word containing an odd number of $y$'s. So one order-2 generator is the central element $x^3$. Note that not all  order-2 generators can be $x^3$, so the remaining two are of the form listed in \eqref{eq:ord-2}. Since $x^3$ is central, when ensuring that the four generators can be multiplied together in some order to equal the identity, we see that the total exponent in $x$ among the two generators from \eqref{eq:ord-2} must be odd. So the possibilities are $\{x^{2k}y,x^{2\ell +1}y\}$, for some $k,\ell\in\{0,1,2\}$.

Consider a surface of genus 3. In light of what we have mentioned above, the only possible ramification types  to consider are $[2,2,2,6]$ and $[2,2,3,3]$. We claim that the latter is not possible. Indeed, for  $[2,2,3,3]$, it follows that both elements of order 2 must come from \eqref{eq:ord-2}. In order to form a spherical generating set, in both these elements the exponent of $x$ must be odd. However we can never generate $y$ from such a set of elements. For $[2,2,2,6]$, one of the elements of order 2 must be $x^3$, as otherwise the total exponent of $y$ in any product of the 4 generators will be odd. It follows that for the remaining two elements of order~2, they are from \eqref{eq:ord-2} and either both $x$-exponents are even, or both are odd. We now compare all the above with the one of type $[2,2,2,2,2,2]$ associated to a surface of genus 7. Likewise, not all generators here can be $y$ or $x^3$, so one is of the form $x^k y$, with $k\ne 0$. Without loss of generality, we may assume that none of the generators is $x^3$, as then condition (\ref{sigmacondition}) does not hold. Recall that there are two conjugacy classes for the non-central elements of order~2. They are $\{y,x^2y,x^4y\}$ and $\{xy,x^3y,x^5y\}$. Observe that, under our current assumptions, in the generating system of  type $[2,2,2,2,2,2]$, not all generators can come from the same conjugacy class, as otherwise it will not be a generating set for $D_6$; indeed, the element $x$ cannot be generated.  Hence we see from the above that condition (\ref{sigmacondition}) does not hold, and hence the only possibilities for the vector of genera are $(2,\ldots,2,3,3,4)$ or   $(2,\ldots,2,4,5)$.

So we now return to the general case, with no  further assumption  on $G$. Let us next assume $g=5$. Then we have $\alpha(A)=3$ and hence $A=[3,3,3,3]$. As a direct check shows that all elements of order~3 in~$G$ generate a proper subgroup, there is no action of $G$ on a surface of genus~5. 

Finally, we consider the remaining case  $(2,\ldots,2,3,3,4)$ for the vector of genera. First assume $G\ne D_6$ and suppose now that $g=4$. Then we have $\alpha(A)=4$ and hence $A=[2,2,2,2,2]$ or $A=[2,2,4,4]$. As indicated above, the first type is not possible, so we consider the second type  $A=[2,2,4,4]$. Clearly this is not possible for $G=\ZZ_2\times \ZZ_6$. For $G=D_{4,3,-1}=\langle x,y\mid x^4=y^3=1,\,xyx^{-1}=y^{-1}\rangle$, the only element of order~2 is~$x^2$ which is central, and the elements of order~4 are
\[
x,\,x^{-1},\,xy,\,xy^{-1},\,x^{-1}y,\,x^{-1}y^{-1}.
\]
Since the product of all elements in the system is the identity, this surmounts to the product of the two order-4 elements being trivial. Hence we cannot get a generating set. Now we only need to consider $G=D_6$. Consider a surface of genus 3. From the above, we have only the ramification type $[2,2,2,6]$, where one of the elements of order 2 must be $x^3$, and for the remaining two elements of order~2, they are from \eqref{eq:ord-2} and either both $x$-exponents are even, or both are odd. For genus $g=4$, we can only have $A=[2,2,2,2,2]$, and here equally one generator has to be $x^3$, as otherwise the total  exponent of $y$ in any product of the 5 generators will be odd. Hence (\ref{sigmacondition}) does not hold, so this final case cannot occur.

\medskip
 
\noindent\underline{Case (c):} $G=D_3$. In this case we can argue as in Case (b), and we only need to consider the vector $(2,\ldots,2,3,4)$ of genera. If the genus is 3, then $\alpha(A)=3$ and $A=[3,3,3,3]$. As $D_3$ cannot be generated by elements of order~3, there is no action on a surface of genus~3. Hence we are done in this case.

\medskip
 
\noindent\underline{Case (d):} $|G|=8$; i.e. $G=Q_8$ or $G=D_4$.  The only vectors of genera to be considered are  $(2,\ldots,2,3,5)$ or $(2,\ldots,2,3,3,3)$. As $Q_8$ has only one element of order~$2$, which is central, condition \eqref{sigmacondition} does not hold, as in every possible ramification type there is an element of even order. So we now set $G=D_4$. For genus 3, we have $\alpha(A)=4$, and so $A=[2,2,2,2,2]$ or $[2,2,4,4]$. Similarly, for genus 2, we have $A=[2,2,2,4]$ (as no product of 3 elements of order 4 can be trivial), and for genus 5, we obtain $A=[2,2,2,2,2,2]$. Note also that for $[2,2,2,2,2]$, at least one generator is central, as otherwise the total $y$-exponent in any product of the 5 generators is odd. Since the square of an element of order 4 is central, we see that the vector $(2,\ldots,2,3,3,3)$ of genera does not satisfy condition (\ref{sigmacondition}). So for the remaining case, it suffices to assume that in a generating system of type  $[2,2,2,2,2,2]$, none of the generators are central; i.e. every generator is conjugate to $y$ or to $xy$  since there are only two conjugacy classes of non-central elements of order 2 in $D_4$. Furthermore, at least one generator is conjugate to~$y$ and another generator is conjugate to~$xy$, otherwise the system would not generate the group. We claim that  condition (\ref{sigmacondition}) is not satisfied. Indeed, for the generating system of type $[2,2,4,4]$, for the product of the generators, taken in some order to be trivial, the system must be of the corresponding form $[x^{2k+m}y,x^{2\ell +m}y,x^\epsilon,x^\delta]$ for some $k,\ell,m\in\{0,1\}$ and $\epsilon,\delta\in\{\pm 1\}$. Similarly, for genus~2, we have that the corresponding system is of the form $[x^{2k+m}y,x^{2\ell +(m+1)}y,x^{2},x^\epsilon]$ for some $k,\ell,m\in\{0,1\}$ and $\epsilon\in\{\pm 1\}$.
Finally returning to the system of type $[2,2,2,2,2]$, note that in order to be a generating system at least one generator is conjugate to~$y$ and another generator is conjugate to~$xy$.  Hence the result follows.

\medskip
 
\noindent\underline{Case (e):} $G= D_{2,8,3}$, the semidihedral group of order $16$. The  vectors of genera to consider here are $(2,\ldots,2,3,9)$, $(2,\ldots,2,5,5)$, $(2,\ldots,2,3,3,5)$ or $(2,\ldots,2,3,3,3,3)$. This leads to type $A=[2,2,2,2,2,2]$ for genus 9, to the types $A\in\{[2,2,2,2,2], [2,2,4,4]\}$   for genus~5, to $A\in\{[2,2,2,4],[2,8,8],[4,4,4]\}$ for genus 3, and $A=[2,4,8]$ for genus 2. Also $G$ cannot be generated by only elements of order $2$, as the only non-central order-2 elements are
\[
x,\,xy^2,\,xy^4,\,xy^6,
\]
and the only central element of order~2 is~$y^4$. So for all remaining options, for any such system of generators, there is a generator of order at least 4. The square of any order-4 generator will be $y^4$, so that the condition (\ref{sigmacondition}) is always violated. 

\medskip
 
\noindent\underline{Case (f):} $|G|=24$; i.e. $G= Z_3\rtimes D_4$ or $G=\text{SL}_2(3)$.  Analogous to the above, the possible genera $(g(C_1),\ldots,g(C_n))$ to be considered are 

\[
(2,\ldots,2,3,13), \,(2,\ldots,2,4,9),\,(2,\ldots,2,5,7),
\]
and
\[
(2,\ldots,2,3,4,5),\,(2,\ldots,2,3,3,7), \,(2,\ldots,2,3,3,3,4).
\]

Since $G$ acts on the surface $C_1$ of genus 2, we have $\alpha(A_1)=24$ and $A_1\in\{[2,4,6],[3,3,4]\}$; indeed, the group $G$ has no element of order 12.

First let $G=Z_3\rtimes D_4$.  We write
\[
G=\langle x,y,z,w\mid x^2=y^2=z^2=w^3=[x,y]=[y,z]=[y,w]=[z,w]=1,\,z^x=zy,\,w^x=w^2\rangle.
\]
For convenience, we list the non-trivial conjugacy classes  of $G$ below:
 \begin{align*}
     \text{ccl}(x)&=\{x,xy,xw,xw^2,xyw,xyw^2\}\\
     \text{ccl}(y)&=\{y\}\\
     \text{ccl}(z)&=\{z,yz\}\\
     \text{ccl}(w)&=\{w,w^2\}\\
     \text{ccl}(xz)&=\{xz,zx,xzw,xzw^2,zxw,zxw^2\}\qquad\text{(order 4)}\\
     \text{ccl}(yw)&=\{yw,yw^2\}\\
     \text{ccl}(zw)&=\{zw,yzw^2\}\\
      \text{ccl}(zw^2)&=\{zw^2,yzw\}
 \end{align*}
 Now by Broughton's classification \cite[Table 5]{Broughton91}, we note that $G$ does not act on a surface of genus 3. Hence here we only need to consider the two vectors of genera $(2,\ldots,2,4,9)$ and $(2,\ldots,2,5,7)$. We begin by observing that if $g=9$ then $A=[3,3,3,3]$. As $[3,3,3,3]$ can never correspond to a generating set for $G$, there is no action on a surface of genus~9. So to settle this case where $G=Z_3\rtimes D_4$, we are left to consider the  genera vector $(2,\ldots,2,5,7)$. We show that \eqref{sigmacondition} is never satisfied by showing that there is always a generator conjugate to~$x$.  Suppose  $G$ acts on a surface of genus 5. We obtain  $A\in\{[2,2,2,6], [2,2,3,3],[3,3,6] \}$. A direct check shows that no generating set can be of type $[3,3,6]$, since $x$ cannot be generated.  Next we observe  that  a generating system of type $[2,2,2,6]$ must have a generator conjugate to~$x$, and similarly for $[2,2,3,3]$.  Next we consider an action on a surface of genus 7. Here $A\in\{[2,2,2,2,2], [2,2,4,4] \}$.  For the type $[2,2,2,2,2]$ it is clear that one generator is conjugate to~$x$. For $[2,2,4,4]$, this is also the case, as if not, either both generators of order 2 are $y$, or both are conjugate to~$z$. In either case, we cannot generate the whole group with such a system. For $g=2$, the only possible option for $A$ to correspond to a generating set is $[2,4,6]$. As the total exponent of $x$ in any word representing a given element of~$G$ is constant, and the  total exponent in an element of order~6 is~0, it follows that the generator of order~2 is conjugate to~$x$.  So \eqref{sigmacondition} does not hold. This finishes the proof for $G=Z_3\rtimes D_4$.

So it remains to consider the group $G=\text{SL}_2(3)$. Recall that there is only one element of order 2 in $\text{SL}_2(3)$, which is central. Hence we deduce that condition (\ref{sigmacondition}) does not hold for all genera vectors, except possibly the vector 
$(2,\ldots,2,4,9)$ since for $g=9$, as there is a possible ramification type with no even element orders, namely $[3,3,3,3]$. Then looking at the ramification types for $g=4$, we only need to consider $A=[4,4,4]$, as the other options either clearly do not give rise to a generating set for the group. However as $\text{SL}_2(3)$ has only one Sylow 2-subgroup, it is now clear that the type $[4,4,4]$ cannot occur.

\medskip
 
\noindent\underline{Case (g):} $G=\text{GL}_2(3)$, which is of order 48. The possible genera $(g(C_1),\ldots,g(C_n))$ that we need to consider are \[
(2,\ldots,2,3,25), \,(2,\ldots,2,4,17),\,(2,\ldots,2,5,13),(2,\ldots,2,3,3,13),\, (2,\ldots,2,7,9),
\]
\[
(2,\ldots,2,3,5,7),\,(2,\ldots,2,3,3,3,7), \,(2,\ldots,2,3,4,9), \,(2,\ldots,2,3,3,4,5), 
\]
\[
 (2,\ldots,2,4,5,5),\,(2,\ldots,2,3,3,3,3,4).
\]
We proceed as before with our analysis on a case-by-case basis. First we collect together some facts about $G$:
\begin{enumerate}
\item [$\bullet$] $G$ has no element of order 12;
    \item [$\bullet$] $G$ has only one conjugacy class of non-central elements of order 2;
           \item [$\bullet$] all elements of order 3, 4 and 6 lie in the subgroup $\text{SL}_2(3)$ of $G$.
\end{enumerate}

From the above facts, we have the following possible ramification types listed according to the genus of the surface:

\begin{center}
\begin{tabular}{|c||c|}
		\hline
	Genus of the surface	 & Possible ramification types\\
		\hline
		$g=2$ & $[2,3,8]$\\
		\hline
  	$g=3$ & $[2,4,6]$\\
		\hline
  	$g=4$ & $[2,4,8]$\\
		\hline
  	$g=5$ & $[2,2,2,3]$, $[2,6,6]$\\
		\hline
  	$g=7$ & $[2,2,2,4]$, $[2,8,8]$\\
		\hline
  $g=9$ & $[2,2,2,6]$, $[2,2,3,3]$\\
		\hline
  $g=13$ & $[2,2,2,2,2]$, $[2,2,4,4]$  \\
		\hline
  $g=17$ & no possible type\\
		\hline
  $g=25$ & $[2,2,2,2,2,2]$\\
		\hline
	\end{tabular}
\end{center}

\medskip

Thus, the genera vector $(2,\ldots,2,4,17)$ is not possible. Condition (\ref{sigmacondition}) is never satisfied for all the possible vectors of genera because in each system of generators of the types listed above, there is at least one non-central generator of order 2. Indeed, otherwise either all elements lie in $\text{SL}_2(3)$, or the system is not a spherical generating system. In particular, for the type $[2,3,8]$, if the generator of order 2 is central, then the product of the generator of order 2 and the generator of order 3 yields an element of  $\text{SL}_2(3)$, so it cannot be an element of order 8 since $\text{SL}_2(3)$ has no element of order~8. For the type $[2,8,8]$, again if the generator of order 2 is central, then all three generators lie in a cyclic subgroup of order~8.  
\end{proof}
 
\begin{cor}
	\label{nleq6}
Let $G$ be a finite group that acts absolutely faithfully on a product $C_1\times\cdots\times C_n$ such that $X=(C_1\times\cdots\times C_n)/G\in \mathcal S_0((-2)^n,n)$. Then $n\leq 6$.
\end{cor}

\begin{proof}
Let $T_1,\ldots,T_n$ be the different spherical systems of generators of $G$ defining the structure of a variety isogenous to a higher product on $X$ and let $A_1,\ldots,A_n$ be the tuples 
of orders of the corresponding elements. By Corollary \ref{orderG} and Lemma \ref{Theta+alpha} we have  
\[
|G|=\prod_{i=1}^n (g(C_i)-1)=\alpha(A_1)(g(C_1)-1)=\cdots=\alpha(A_n)(g(C_n)-1).
\]
It follows that
\[
|G|^n=\alpha(A_1)\cdots\alpha(A_n)\prod_{i=1}^n(g(C_i)-1)=|G|\prod_{i=1}^n\alpha(A_i).
\]
Equivalently 
\[
|G|^{n-1}=\prod_{i=1}^n\alpha(A_i).
\]
For each $i\in\{1,\ldots,n\}$ we have $\alpha(A_i)\leq 84$ (see Corollary \ref{Listoforders}). This implies $|G|^{n-1}\leq 84^{n}$, i.e. $|G|\leq 84^{n/(n-1)}$. On the other hand, as $g(C_i)>2$ for every $i\in\{1,\ldots,n\}$ from Proposition~\ref{genus>2}, we have $|G|=\prod_{i=1}^n(g(C_i)-1)\geq 2^n$. Altogether we obtain $2<84^{1/(n-1)}$. The last expression is decreasing in $n$ and we have $84^{1/7}=1.88\ldots$. This proves that the dimension of $X$ is less or equal to $7$. In the case $n=7$ we must have $g(C_i)=3$ for all $i\in\{1,\ldots,7\}$. Otherwise we would have $192=2^6\cdot 3\leq |G|\leq 84^{7/6}=175.7\ldots$. However there is no group of order $2^7=128$ that will satisfy Corollary \ref{Listoforders}. Indeed, if this were the case we would find a spherical system of generators $T$ such that the corresponding ramification type $A$ satisfies $\alpha(A)=2^6=64$. But there is no such type in the table of Corollary \ref{Listoforders}.   
\end{proof}

\begin{prop}\label{prop:n=6}
The set $\mathcal S_0(2^6,6)$ is empty.
\end{prop}

\begin{proof}
Let $X=(C_1\times\cdots\times C_6)/G$ be in $\mathcal S_0(2^6,6)$. From the proof of Corollary \ref{nleq6} we have $2^6\leq |G|\leq 84^{6/5}\approx 203$. We also note that at least one factor $C_i$ has genus $3$, as otherwise $|G|=\prod_{i=1}^6 (g(C_i)-1)\geq 3^6=729>203.$ The classification in \cite[Table 5]{Broughton91} gives only two such groups, namely $G=(Z_4\times Z_4)\rtimes S_3$ of order $96$ and $G=\text{GL}_3(2)$ of order $168$. Note that $96=2^5\cdot 3$, so if $G=(Z_4\times Z_4)\rtimes S_3$ defines an element of $\mathcal S_0(2^6,6)$, then there must be an action of $G$ on a surface of genus $4$ with quotient of genus zero. However there is no such action, since it must correspond with a ramification type~$A$ with $\alpha(A)=32$ that does not appear in Corollary \ref{Listoforders}. For the other candidate $G=\text{GL}_3(2)$  note that the order $168=2^3\cdot 3\cdot 7$ cannot be written as a product of $6$ different numbers. 
\end{proof}

\begin{proof}
    [Proof of Theorem \ref{mainthm}]
    The first statement is just Proposition \ref{genus>2}. Then Corollary \ref{nleq6} and Proposition \ref{prop:n=6} give the final statement.
\end{proof}

\section{Possible orders for $G$} \label{sec:5}

For the remainder of the paper, the group $G$ will always be such that $(C_1\times\cdots\times C_n)/G\in \mathcal S_0((-2)^n,n)$, where from the previous section $n\in\{4,5\}$. This will form our standing assumptions in this last section. We first prove several lemmas which narrow down the possible orders for $G$. All these restrictions will help us prove Theorem~\ref{thm:main}, which we do at the very end. 

\begin{lemma}\label{lem:divisors-options}
The prime divisors of the order of $G$ form a subset of $\{2,3,5,7\}$.
\end{lemma}

\begin{proof}
    From Corollary~\ref{orderG} and Lemma~\ref{Theta+alpha}(iii), it follows that a prime divisor of $|G|$ has to be a prime divisor of some $\alpha(A_i)$. From Corollary~\ref{Listoforders}, the prime divisors of an $\alpha(A_i)$ can only be a subset of $\{2,3,5,7\}$.
\end{proof}

\begin{lemma}\label{lem:max-divisors}
The order of $G$ has at most 3 distinct prime divisors.
\end{lemma}

\begin{proof}
    Suppose for a contradiction that $|G|$  has 4 distinct prime divisors, which by the previous lemma are then  2, 3, 5 and 7. Since from Corollary~\ref{Listoforders} no $\alpha(A_i)$ is divisible by $5^2$ or $7^2$, it follows that $|G|=2^a\cdot 3^b\cdot 5\cdot 7$ for some $a,b\in\mathbb{N}$. From  Corollary~\ref{Listoforders}, the only options for $\alpha(A_i)$ with 7 as a factor are $2^2\cdot 3\cdot 7$ and $3\cdot 7$. However, since $n\ge 4$, the only option here for such an $\alpha(A_i)$ is $2^2\cdot 3\cdot 7$. Appealing to Lemma~\ref{Theta+alpha}(iii), there then exists one $j\in\{1,\ldots,n\}$, such that $\alpha(A_i)=2^2\cdot 3\cdot 7$ for all $i\ne j$, and  $7\nmid\alpha(A_j)$. However, since by assumption $5\mid |G|$, we likewise have that there is a unique $k\in\{1,\ldots,n\}$ such that $5\mid \alpha(A_\ell)$ for all $\ell\ne k$ and $5\nmid \alpha(A_k)$. Since $n\ge 4$, we then have that there is an $i\ne j,k$ such that $5\mid \alpha(A_i)$, which yields the desired contradiction. 
\end{proof}

\begin{lemma}\label{lem:3-divisors}
Suppose that  $|G|$ has  3 distinct prime divisors. Then $n=4$ and for distinct $i_1,i_2,i_3,i_4\in\{1,2,3,4\}$ either
\begin{enumerate}
    \item [\textup{(i)}] $|G|=60$ with $\alpha(A_{i_1})=12,\alpha(A_{i_2})=\alpha(A_{i_3})=30, \alpha(A_{i_4})=20$;
    or
    \item [\textup{(ii)}] $|G|=168$ with $\alpha(A_{i_1})=8,\alpha(A_{i_2})=\alpha(A_{i_3})= \alpha(A_{i_4})=84$.
\end{enumerate}  
\end{lemma}

\begin{proof}
    From Lemma~\ref{lem:divisors-options}, the set of 3 distinct prime divisors of $|G|$ can either be $\{2,5,7\}$, $\{2,3,5\}$ or $\{2,3,7\}$. The case $\{2,5,7\}$ is clear, since from Corollary~\ref{Listoforders} there is no $\alpha(A_i)$ which is divisible by $35$.  For the case $\{2,3,5\}$,  Corollary~\ref{Listoforders} and Lemma~\ref{Theta+alpha}(iii) show that for one $i\in\{1,\ldots,n\}$, without loss of generality say $i=1$, we have $5\nmid \alpha(A_1)$ but $5\mid \alpha(A_i)$ for all $i\ne 1$. Since $n\ge 4$, the only options for $\alpha(A_i)$, for $i\ne 1$, are $2^3\cdot 5$, $2\cdot 3\cdot 5$ or $2^2\cdot 5$. Repeating the same considerations with the prime $3$, we deduce that for at least one $j\in\{2,\ldots,n\}$, say $j=2$, we have $\alpha(A_2)=2\cdot 3\cdot 5$. If we have $\alpha(A_j)=2^2\cdot 5$ for some $j>2$, say $j=3$, then it follows that $|G|=2^3\cdot 3\cdot 5$ or $|G|=2^2\cdot 3\cdot 5$. Since $n\ge 4$, only the latter option is possible. Indeed, as mentioned above  for all $i\ne 1$, we have $\alpha(A_i)\in\{2^3\cdot 5, 2\cdot 3\cdot 5,2^2\cdot 5\}$ and each $\alpha(A_i)$ must be a product of three out of $n$ fixed factors of $|G|$. So $n=4$ and altogether we have
    \[
    \alpha(A_1)=2^2\cdot 3,\quad\alpha(A_2)=2\cdot 3\cdot 5,\quad  \alpha(A_3)=2^2\cdot 5,\quad \alpha(A_4)=2\cdot 3\cdot 5.
    \]
    Note that $\alpha(A_3)=2^3\cdot 5$ is not possible, as by Corollary~\ref{orderG} and Lemma~\ref{Theta+alpha}(iii), we must then have $$2\cdot 3\cdot 5\cdot x=\alpha(A_2)\cdot x=\alpha(A_3)\cdot y=2^3\cdot 5\cdot y$$ for some $x,y\in\mathbb{N}$ with $x$ being a non-trivial power of 2 and $y\mid 2\cdot 3$. It follows that $x=2^{2+a}$, with $a\ge 0$ and $y=2^a\cdot 3$. Since $y\mid 2\cdot 3$, we have $a=1$. However, then here $|G|$ is factorised as $8\cdot 5\cdot 6$, and so we   cannot have $n\ge 4$. Similarly, the remaining option of $\alpha(A_2)=\cdots=\alpha(A_n)=2\cdot 3\cdot 5$ is not possible for $n\ge 4$.

    Lastly, we suppose that the prime divisors of $|G|$ are 2, 3 and 7. As before, we cannot have $7^2\mid |G|$. Without loss of generality we suppose that $7\nmid \alpha(A_1)$. Therefore, since $n\ge 4$, Corollary~\ref{Listoforders} yields that 
    \[
    \alpha(A_2)=\cdots=\alpha(A_n)=2^2\cdot 3\cdot 7,
    \]
    and hence $|G|=2^a\cdot 3\cdot 7$, with $a\ge 3$, and furthermore $\alpha(A_1)=2^a$. Similarly, appealing to Corollary~\ref{orderG} and Lemma~\ref{Theta+alpha}(iii), it follows that $a=3$ and hence $n=4$.
\end{proof}

\begin{lemma}\label{lem:2-divisors}
Suppose that   $|G|$  has  2 distinct prime divisors $p$ and $q$. Then $\{p,q\}$ is either $\{2,5\}$ or $\{2,3\}$. 

Furthermore, if $\{p,q\}=\{2,5\}$, then either 
\begin{enumerate}
    \item [\textup{(i)}] $|G|=40$ and     $n=4$; or
    \item [\textup{(ii)}] $|G|=80$; 
    \end{enumerate}
 and if $\{p,q\}=\{2,3\}$, then either 
 \begin{enumerate}
    \item [\textup{(iii)}] $|G|\in\{24,\,36,  \,  72\}$     and $n=4$; or
    \item [\textup{(iv)}] $|G|=48$. 
    \end{enumerate}

\end{lemma}

\begin{proof}
    From Lemma~\ref{lem:divisors-options} together with Corollary~\ref{Listoforders}, we observe that the set of 2 distinct prime divisors of $|G|$ can either be $\{2,3\}$, or $\{2,5\}$.

\medskip

\noindent\underline{Case 1:} Suppose $|G|=2^a5^b$ for some $a,b\in\mathbb{N}$.
As seen before, we may assume that $5\nmid \alpha(A_1)$ but $5\mid \alpha(A_i)$ for all $i\ne 1$. Since $n\ge 4$, the options for $\alpha(A_i)$, for $i\ne 1$, are $2^3\cdot 5$ or $2^2\cdot 5$, and the options for $\alpha(A_1)$ are $2^3$ or $2^4$.

Suppose $\alpha(A_1)=2^3$. The only possibility is 
\[
\alpha(A_1)=2^3,\quad \alpha(A_2)=2^2\cdot 5,\quad \alpha(A_3)=2^2\cdot 5,\quad \alpha(A_4)=2^2\cdot 5,
\]
and here $|G|=2^3\cdot 5$.

Next suppose $\alpha(A_1)=2^4$. Then the only two combinations that work are
\[
\alpha(A_1)=2^4,\quad \alpha(A_2)=2^3\cdot 5,\quad \alpha(A_3)=2^3\cdot 5,\quad \alpha(A_4)=2^2\cdot 5,
\]
 as well as 
\[
\alpha(A_1)=2^4,\quad \alpha(A_2)=\alpha(A_3)=\alpha(A_4)=\alpha(A_5)=2^3\cdot 5.
\]
In both situations we have $|G|=2^4\cdot 5$. 

\medskip

\noindent\underline{Case 2:} Suppose $|G|=2^a3^b$ for some $a,b\in\mathbb{N}$.
First we suppose that $3^2\mid |G|$. Then without loss of generality, we may suppose that $3^2$ divides both $\alpha(A_2)$ and $\alpha(A_3)$. From Corollary~\ref{Listoforders}, we have that the options for $\alpha(A_2)$ and $\alpha(A_3)$ are $2^2\cdot 3^2$ or $2\cdot3^2$. Further, from considering Corollary~\ref{orderG} and Lemma~\ref{Theta+alpha}, we note that $\alpha(A_2)=\alpha(A_3)$. Similarly, one sees that the only combinations (subject to reordering the $A_i$'s) are:
\begin{enumerate}
    \item [(a)] $\alpha(A_1)=2^2\cdot 3$, $\alpha(A_2)=2\cdot 3^2$,  $\alpha(A_3)=2\cdot 3^2$, $\alpha(A_4)=2^2\cdot 3$, and so $|G|=2^23^2$;
    \item [(b)] $\alpha(A_1)=2^3$, $\alpha(A_2)=2^2\cdot 3^2$,  $\alpha(A_3)=2^2\cdot 3^2$, $\alpha(A_4)=2^2\cdot 3^2$, and so $|G|=2^33^2$.
\end{enumerate}
In particular, here $n=4$.

Next we suppose that $3^2\nmid |G|$. Hence, assuming that $3\nmid \alpha(A_1)$, the options for $\alpha(A_i)$, for $i\ne 1$, are $2^4\cdot 3$, $2^3\cdot 3$, or $2^2\cdot 3$. As seen before, the only options for $\alpha(A_1)$ are $2^4$ or $2^3$. Then the four possible combinations (again subject to reordering the $A_i$'s) are
\[
\alpha(A_1)=2^4,\quad \alpha(A_2)=2^3\cdot 3,\quad \alpha(A_3)=2^3\cdot 3,\quad \alpha(A_4)=2^2\cdot 3,
\]
or
\[
\alpha(A_1)=2^4,\quad \alpha(A_2)=2^3\cdot 3,\quad \alpha(A_3)=2^3\cdot 3,\quad \alpha(A_4)=2^3\cdot 3,\quad \alpha(A_5)=2^3\cdot 3,
\]
or
\[
\alpha(A_1)=2^3,\quad \alpha(A_2)=2^3\cdot 3,\quad \alpha(A_3)=2^3\cdot 3,\quad \alpha(A_4)=2^3\cdot 3,
\]
and thus $|G|=2^4\cdot 3$, or
\[
\alpha(A_1)=2^3,\quad \alpha(A_2)=2^2\cdot 3,\quad \alpha(A_3)=2^2\cdot 3,\quad \alpha(A_4)=2^2\cdot 3,
\]
and here $|G|=2^3\cdot 3$.  Hence the result.
\end{proof}

\begin{lemma}\label{lem:80}
Suppose that $|G|$  is divisible by only $2$ and $5$. Then  $|G|=40$.
\end{lemma}

\begin{proof}
  By the previous result, we only need to consider groups of order $80$. For a group of order 80, the above proof yields that we need an action on a surface of genus 3.  Here we have $\alpha(A)=40$, and so $A=[2,3,15]$ by Corollary~\ref{Listoforders}. But since $|G|=80$, there is no element of order 3.
\end{proof}

Alternatively, the above result follows immediately from \cite[Table~5]{Broughton91}, which shows that the automorphism group of a surface of genus 3 never has order 80. Similarly, many of the proofs below can be shortened using Broughton's classification in \cite[Table~5]{Broughton91} or even using an ingredient of Broughton's classification that the possible orders of automorphisms of a surface of genus 3 are 2, 3, 4, 6, 7, 8, 9, 12, or 14. Though we provide  alternative self-contained proofs, when this is feasible. 

\begin{lemma}\label{lem:one-prime-divisor}
Suppose that  $|G|$ is a prime power. Then either
\begin{enumerate}
    \item [\textup{(i)}] $|G|=2^4$ and     $n=4$; or
    \item [\textup{(ii)}] $|G|=2^5$. 
\end{enumerate}  
\end{lemma}

\begin{proof}
    Suppose $|G|=p^a$ for some prime $p$ and for $a\ge 4$. From Corollary~\ref{Listoforders}, we see that $p=2$, since $3^2$ is the highest prime power listed, which is not a power of 2. Since the highest power of 2 listed is $2^4$, one can check that $|G|\ge2^6$ is not possible if $n\ge 4$. The remaining cases (as usual subject to reordering the $A_i$'s) are all possible, with the following breakdown of values:
    \begin{enumerate}
        \item [(a)] $\alpha(A_1)=2^3$, $\alpha(A_2)=2^3$,  $\alpha(A_3)=2^3$, $\alpha(A_4)=2^3$, and so $|G|=2^4$;
        \item [(b)] $\alpha(A_1)=2^4$, $\alpha(A_2)=2^4$,  $\alpha(A_3)=2^4$, $\alpha(A_4)=2^3$, and thus $|G|=2^5$.
        \item [(c)] $\alpha(A_1)=2^4$, $\alpha(A_2)=2^4$,  $\alpha(A_3)=2^4$, $\alpha(A_4)=2^4$, $\alpha(A_5)=2^4$, and  $|G|=2^5$.\qedhere
    \end{enumerate}
\end{proof}

\begin{lemma}\label{lem:abelian}
The group $G$ is non-abelian.
\end{lemma}

\begin{proof}
  Suppose for a contradiction that $G$ is abelian. For the case when $|G|$ has three distinct prime divisors, it follows immediately from Corollary~\ref{Listoforders} that this is not possible, because, referring to Lemma~\ref{lem:3-divisors}, the types associated to $\alpha(A_{i_1})$ do not contain an element of order divisible by 5 or 7. 

  We next consider the case when $|G|$ is a prime power, that is $2^4$ or $2^5$. For the case $|G|=2^5$, note that for some~$i$   the type for  $A_i$ is $[2,4,8]$.   However $[2,4,8]$ cannot be the type of a spherical system of generators for an abelian group.   For the case $|G|=2^4$, the possible types  are $[2,2,2,4]$, $[2,8,8]$, or $[4,4,4]$. Note that $[2,2,2,4]$ is not possible for a spherical system of generators, since we need at least two elements of order 4, in order for the product of all four elements to be trivial. If for one $i$ we have  $A_i=[2,8,8]$, then $G\cong Z_2\times Z_8$. It follows that $A_i=[2,8,8]$ for every $i$, and then condition (\ref{sigmacondition}) does not hold.  So we may assume that  $A_i=[4,4,4]$ for every $i$. Thus $G\cong Z_4\times Z_4$. We likewise show that condition (\ref{sigmacondition}) does not hold as follows. Let $G=\langle x,y\rangle$. First observe that the elements of order~2 are exactly $x^2,y^2,x^2y^2$. Next, for any spherical generating system of type $[4,4,4]$, exactly one generator must be from the coset $x^{\pm 1}G^2=x^{\pm 1}\Phi(G)$. Hence we see that $x^2\in\cap_{i=1}^4 \Sigma(T_i)$.

  Finally we consider the case when $|G|$ has two distinct prime divisors. From Lemma~\ref{lem:80}, if 5 divides $|G|$, we only need to consider the abelian groups of order $40$. 
  From the proof of Lemma~\ref{lem:2-divisors}, we have that  $\alpha(A_1)=8$ and as $G$ must have a generator of order divisible by 5, the result follows from Corollary~\ref{Listoforders}.  We next consider abelian groups of order 36 and of order 72; cf. Lemma~\ref{lem:2-divisors}. For $|G|=72$, from the proof of Lemma~\ref{lem:2-divisors}, we have that  $\alpha(A_i)=36$, for some $i$, and for $|G|=36$ we likewise have that  $\alpha(A_i)=18$.  Hence for $|G|=72$,  Corollary~\ref{Listoforders} yields that   $A_i=[2,3,9]$, which cannot correspond to a spherical system of generators for~$G$. For $|G|=36$ we have  $A_i=[2,3,18]$, which similarly never yields a spherical system of generators for $G$. For the remaining groups of orders 24 and 48, note that if $\alpha(A_i)=8$ there are no feasible options.   It then suffices to consider the case when $\alpha(A_i)=16$, which has the type $[2,4,8]$. As we clearly cannot have a generating set here, the proof is complete.  
\end{proof}

\begin{thm}\label{thm:order-168}
The order of $G$ has at most $2$ distinct prime divisors.
\end{thm}

\begin{proof}
    By Lemmata~\ref{lem:max-divisors} and \ref{lem:3-divisors}, we just need to consider groups of order 60 and 168. From Lemma~\ref{lem:3-divisors}, we see that a group of order 168 must then act on a surface of genus~3 and on a surface of genus 22. For a surface of genus 22, we have $\alpha(A)=8$ and so $A\in\{[2,2,2,4], [2,8,8],    [4,4,4]\}$. We recall that there are 57 isomorphism types for  groups of order 168.     There are 56 isomorphism types of groups of order 168 which are solvable     groups, and the final isomorphism type is the simple group $\text{PSL}_2(7)$. 
    
    For the solvable groups~$G$, recall that either the Sylow 3-subgroup is normal or the Sylow 7-subgroup is normal. Indeed, a solvable group has an abelian normal subgroup $N$ of prime order. We may suppose $|N|\ne 3,7$, else we are done.  So $|N|=2$, and from considering $G/N$, we deduce that the Sylow 7-subgroup of $G$ is normal.  Suppose first that the Sylow 7-subgroup is normal. Consider the quotient of $G$ by its Sylow 7-subgroup. As a spherical system of generators for~$G$ is still a spherical system of generators in the quotient, looking at the ramification type $[2,3,7]$, which is the only possibility for an action on a surface of genus 3, if $G$ had a spherical generating system $(g_1,g_2,g_1g_2)$ such that $\{\text{o}(g_1),\text{o}(g_2),\text{o}(g_1g_2)\}=\{2,3,7\}$, then, in the quotient of $G$ by its Sylow 7-subgroup, we would need the generator of order 3 to be the multiplicative inverse of the generator of order~2. Similarly, if the Sylow 3-subgroup were normal, in the corresponding quotient, we would obtain the generator of order 2 being the multiplicative inverse of the generator of order~$7$. Hence, we obtain a contradiction. Thus $G$ here has no action on a surface of genus 3.

    It remains here to consider $G=\text{PSL}_2(7)$.    Recall that the elements of order 2 form one conjugacy class in $\text{PSL}_2(7)$. Hence condition (\ref{sigmacondition}) does not hold.

    Now let $G$ be a group of order 60.     From Lemma~\ref{lem:3-divisors}(i), we must have an action on a surface of genus~3, corresponding to $\alpha(A)=30$ and so $A=[2,3,10]$, and we must have an action on a surface of genus~4, corresponding to $\alpha(A)=20$ and so $A=[2,5,5]$. There are 13 isomorphism classes for a group of order 60, two of which are abelian. These are the cyclic group $Z_{60}$ and $Z_2\times Z_{30}$.  By Lemma~\ref{lem:abelian}, these groups can be excluded. The remaining 11 isomorphism types are then
    \[
    A_5,\, D_{30}, \,Z_3\rtimes \text{Dic}_5,\, Z_5\rtimes \text{Dic}_3,
    \, S_3\times D_5,\,Z_3\times (Z_5\rtimes Z_4),
    \]
    \[
    Z_5\times A_4,\,Z_6\times D_5,\,Z_{10}\times S_3,\, Z_{5}\times\text{Dic}_3,\,Z_{3}\times\text{Dic}_5,
    \]
    where
    \begin{align*}
        \text{Dic}_3&=\langle x,y\mid x^{6}=1,\, y^2=x^3,\, yxy^{-1}=x^{-1}\rangle\\
        \text{Dic}_5&=\langle x,y\mid x^{10}=1,\, y^2=x^5,\, yxy^{-1}=x^{-1}\rangle
    \end{align*}
    and 
    \begin{align*}
    Z_3\rtimes \text{Dic}_5&=\langle a,b \mid a^{30}=1,\, b^2=a^{15},\, bab^{-1}=a^{-1} \rangle\\
       Z_5\rtimes \text{Dic}_3&=\langle  a,b,c \mid a^3=b^5=c^4=1, \,ab=ba, \,cac^{-1}=a^{-1},\, cbc^{-1}=b^3 \rangle,\\
       Z_5\rtimes Z_4&= \langle a,b \mid a^5=b^4=1, \,bab^{-1}=a^3\rangle; 
    \end{align*}
    cf. the database GroupNames~\cite{GroupNames}.

    Now, we will show that there is no generating pairs $(g_1,g_2)$ and $(h_1,h_2)$ of $G$, 
    for each $11$ choices of $G$ above, such that $\{\text{o}(g_1),\text{o}(g_2),\text{o}(g_1g_2)\}=\{2,3,10\}$ and $\{\text{o}(h_1),\text{o}(h_2),\text{o}(h_1h_2)\}=\{2,5,5\}$. For $G=A_5$, as there is no element of order 10, the result is immediate. 

    For $G\in\{Z_3\times (Z_5\rtimes Z_4),\,Z_{3}\times\text{Dic}_5,\,Z_6\times D_5\}$, we focus on the generator of order 3. Since $Z_5\rtimes Z_4$, $D_5$ and $\text{Dic}_5$ have no elements of order 3, we may assume without loss of generality that $g_1$ is of order 3. But then 3 must divide the order of $g_1g_2$. Hence the signature $[2,3,10]$ cannot be attained in these groups. For $G=Z_3\rtimes \text{Dic}_5$, note that the only elements of order 3 and 10 lie in the cyclic subgroup of order 30. Likewise the signature $[2,3,10]$ cannot yield a spherical system of generators.

    For groups $G\in\{Z_5\rtimes \text{Dic}_3,\,Z_5\times A_4,\,Z_{10}\times S_3,\, Z_{5}\times\text{Dic}_3,\,S_3\times D_5\}$, we instead consider an action on a surface of genus 4, where $\alpha(A)=20$ and $A=[2,5,5]$. Let $x$ be an element of order 5. We have that the only generators of order 5 in $G$ are $x$, $x^2$, $x^{3}$ and $x^4$.     So there does not exist a spherical system of generators of $G$ of type  $[2,5,5]$. 
    
    For $G=D_{30}$, writing $a$ for an element of order~30, we likewise have that, for the genus 4 case, we can never have a spherical system of generators of $G$ of type  $[2,5,5]$. This completes the proof. 
\end{proof}

\begin{prop}\label{prop:divisors-2-5}
Suppose that  $|G|$ has  $2$ distinct prime divisors $p$ and $q$. Then $\{p,q\}=\{2,3\}$ and furthermore $|G|\in\{24,48\}$.
\end{prop}

\begin{proof}
   By Lemmata~\ref{lem:2-divisors} and \ref{lem:80}, we only need to consider groups of order $40$  to prove the first statement. For a group of order 40, it must act on a surface of genus 3 and of genus~6; cf. the proof of Lemma~\ref{lem:2-divisors}. It suffices to show that $G$ does not act on a surface of genus~3.  Here we have $\alpha(A)=20$ and so $A=[2,5,5]$. With reference to the database~\cite{GroupNames}, the isomorphism types for the non-abelian groups of order 40 are:
    \[
    D_{20},\,\text{Dic}_{10},\,Z_5\rtimes D_4,\, Z_5\rtimes Z_8,\, Z_5\rtimes_2 Z_8,\,  Z_5\rtimes(Z_2\times Z_4),\,Z_4\times D_5,
    \]
    \[
    Z_5\times D_4,\,Z_2^2\times D_5,\,Z_5\times Q_8,\,Z_2\times\text{Dic}_{5},
 \]
    where 
    \begin{align*}
    \text{Dic}_{10}&=\langle a,b \mid a^{20}=1, b^2=a^{10}, bab^{-1}=a^{-1}\rangle,\\
    Z_5\rtimes D_4&=\langle a,b,c \mid  a^5=b^4=c^2=1,\, bab^{-1}=cac=a^{-1},\, cbc=b^{-1}\rangle,\\
        Z_5\rtimes Z_8&=\langle  a,b \mid a^5=b^8=1,\, bab^{-1}=a^3 \rangle,\\
      Z_5\rtimes_2 Z_8&=\langle a,b \mid a^5=b^8=1,\, bab^{-1}=a^{-1}\rangle,\\
   Z_5\rtimes(Z_2\times Z_4)&= \langle  a,b,c \mid a^2=b^5=c^4=1, \,ab=ba, \,ac=ca, \,cbc^{-1}=b^3  \rangle,\\
    Z_2\times\text{Dic}_{5}&= \langle  a,b,c \mid a^2=b^{10}=1, \,c^2=b^5,\, ab=ba,\, ac=ca, \,cbc^{-1}=b^{-1}\rangle.
    \end{align*} By inspection, in all the groups above there is only one cyclic subgroup of order~5. So type $[2,5,5]$ is not possible for a spherical system of generators.

    For the final statement, it suffices to show that orders 36 and 72 are not possible. For $|G|=36$, the  proof of Lemma~\ref{lem:2-divisors} shows that there must be an  action on a surface of genus~3 and on a surface of genus~4, and for $|G|=72$, on surfaces of genus~3 and of genus~10. For $G$ of order 36, the ramification type corresponding to genus 3 is $[2,3,18]$, and so $G$ must have a normal cyclic subgroup of order~$18$. Since then there is only one subgroup of order~3, which is a cyclic subgroup of the cyclic subgroup of order~18, there is no spherical system of generators of type $[2,3,18]$. For groups of order 72, note that we require an action on a surface of genus~3. However by \cite[Table~5]{Broughton91}, there is no such action. 
\end{proof}

\begin{prop}\label{prop:prime-power-no}
The order of $G$ is not a prime power. 
\end{prop}

\begin{proof}
 By Lemma~\ref{lem:one-prime-divisor}, it suffices to show that $|G|$ cannot be 16 or 32. Suppose first that $|G|=16$. The proof of Lemma~\ref{lem:one-prime-divisor} yields $\alpha(A_i)=8$ for all $i$, and so $G$ acts on a surface of genus~3. From \cite[Table~5]{Broughton91}, up to isomorphism we only need to consider:
\begin{enumerate}
    \item [(a)] $G=D_{2,8,5}=\langle x,y\mid x^2=y^8=1, xyx^{-1}=y^5\rangle$, with ramification type $[2,8,8]$;
      \item [(b)]  $G=D_{4,4,-1}=\langle x,y\mid x^4=y^4=1, xyx^{-1}=y^{-1}\rangle$, with ramification type $[4,4,4]$.
\end{enumerate}
Indeed,  all other options for $G$ listed in \cite[Table~5]{Broughton91} are either abelian or clearly do not satisfy condition~\eqref{sigmacondition}. For (a), it is clear that condition~\eqref{sigmacondition} does not hold, since all cyclic subgroups of order 8 intersect non-trivially. For (b), note that the elements of order 4 are 
\[
x,\,x^3,\,xy^2,\,x^3y^2,\,xy,\,xy^3,\,x^3y,\,x^3y^3,\quad\text{all of which have square $x^2$,}
\]
and 
\[
y,\,y^3,\,x^2y,\,x^2y^3,\quad\text{all of which have square $y^2$.}
\]
Hence we see that \eqref{sigmacondition} does not hold, since all spherical systems of generators with type $[4,4,4]$ must contain a generator which has square $x^2$.

For $|G|=32$, we likewise note from the proof of Lemma~\ref{lem:one-prime-divisor} that $\alpha(A_i)=16$ for some $i$, and thus $G$ acts on a surface of genus~3. From \cite[Table~5]{Broughton91}, up to isomorphism we have the following options only:
\begin{enumerate}
    \item [(a)] $G=Z_2\ltimes (Z_2\times Z_8)=\langle x,y,z\mid x^2=y^2=z^8=[y,z]=[x,y]=1, \,xzx=yz^3\rangle$, with ramification type $[2,4,8]$; 
      \item [(b)] $G=Z_2\ltimes D_{2,8,5}=\langle x,y,z\mid x^2=y^2=z^8=1, \,yzy=z^5,\,xyx=yz^4,\,xzx=yz^3\rangle$, with ramification type $[2,4,8]$. 
\end{enumerate}
In both groups there are only two cyclic subgroups of order~8, namely $\langle z\rangle$ and $\langle yz^3\rangle$, which are clearly conjugate to each other.  Hence we see that condition (\ref{sigmacondition}) does not hold.
\end{proof}

\begin{thm}\label{thm:24}
      The order of $G$ is not $24$.
\end{thm}

\begin{proof}
Suppose otherwise. We see from the proof of Lemma~\ref{lem:2-divisors}(ii) that $G$ must act on a surface of genus~3 and on a surface of genus 4. As given in~\cite{GroupNames}, the non-abelian groups of order 24, up to isomorphism, are the following 12 groups:
    \[
    S_4,\, D_{12},\, \text{Dic}_6,\,\text{SL}_2(3),\,Z_3\rtimes D_4,\,Z_3\rtimes Z_8,
    \]
    \[
   Z_2\times A_4,\,Z_4\times S_3,\,Z_3\times D_4,\,Z_2^2\times S_3,\,Z_3\times Q_8,\,Z_2\times \text{Dic}_3.
    \]
    For a surface of genus 3 on which $G$ acts on, the corresponding $\alpha(A)$ is 12. Hence the possible ramification types for $A$ are
    \[
      [2,2,2,3],\quad [2,4,12],\quad [2,6,6],\quad [3,6,6],\quad [3,4,4].
    \]
    Similarly for a surface of genus 4, the possible ramification types are
    \[
      [2,2,2,4],\quad [2,8,8],      \quad [4,4,4].
    \]

Suppose $G=\text{SL}_2(3)$. Since there is only one element of order 2, we clearly see that regardless of the ramification types for the surfaces of genus 3 and 4, we always have that condition (\ref{sigmacondition}) does not hold, since the intersection of the $\Sigma(T_i)$ contains the element of order~2.

Suppose $G\in\{Z_3\times D_4,\,Z_3\times Q_8\}$. Observe that any generating set for $G$ has to contain an element of order divisible by 3. Furthermore a spherical system of generators for $G$ has to contain at least two generators with orders divisible by 3. Hence there is no action on a surface of genus 4. 

Suppose $G=Z_2\times \text{Dic}_3$. Recall that $\text{Dic}_3=\{1,x,x^2,\ldots,x^5,y,xy,x^2y,\ldots,x^5y\}$. For $\text{Dic}_3$, note that all elements of  order~4 have the form $x^ky$, for $k\in\{0,1,\ldots,5\}$, and the only element of order 2 is central. As $G$ has no elements of order 8 or 12, the possible ramification types for genus~3 are $[2,2,2,3]$, $[2,6,6]$, $[3,6,6]$, or $[3,4,4]$, and the possibilities for genus 4 are $[2,2,2,4]$ or $[4,4,4]$. Since all elements of order~2 are central, the ramification types $[2,2,2,3]$, $[2,6,6]$, $[3,6,6]$ can never correspond to a generating system for $G$. All remaining ramification types contain an element of order 4, whose square is $x^3$. Thus (\ref{sigmacondition}) does not hold.

Suppose $G\in\{Z_2\times A_4,Z_2^2\times S_3\}$, which  has no elements of order 4. Hence there are no possible ramification types for an action of $G$ on a surface of genus 4.

Suppose $G=Z_4\times S_3$. Here $G$ has no element of order 8, so  the possible ramification types for genus 4 are $[2,2,2,4]$ or $[4,4,4]$. However, since the ramification type of a spherical system of generators of $G$ must have an even number of components being divisible by~4, we see that there is no action of $G$ on a surface of genus 4.
    
     For $G=S_4$, since there are no elements of order 6, 8 or 12, the ramification types for genus~3 are
   $[2,2,2,3]$ or $[3,4,4]$, and for  genus 4 they are 
    $[2,2,2,4]$ or $[4,4,4]$.  Since there is only one conjugacy class for the square of an element of order 4, for at least one surface of genus 3, the corresponding ramification type associated to it has to be  $[2,2,2,3]$, otherwise it is immediate that (\ref{sigmacondition}) does not hold. However, for $[2,2,2,3]$ to correspond to a spherical system of generators, since a 3-cycle is an even permutation, at least one of the generators of order 2 has to be a product of two transpositions, which is a square of a 4-cycle. Hence (\ref{sigmacondition}) never holds.

    For $G=Z_3\rtimes Z_8=\langle x,y\mid x^3=y^8=1,\,yxy^{-1}=x^{-1}\rangle$, note that $y^2$ and $y^6$ are the only elements of order 4 and $y^4$ is the only element of order~2.     Hence (\ref{sigmacondition}) never holds for this group.
    
    For $G=Z_3\rtimes D_4=\langle x,y,z\mid x^3=y^4=z^2=1,\,yxy^{-1}=zxz=x^{-1},\, zyz=y^{-1}\rangle$, as the only elements of order~$4$ are $y$ and $y^{-1}$, and as there are no elements in~$G$ of order ~$8$, the ramification type for a surface of genus~$4$ must be $[2,2,2,4]$. Likewise, as the only elements of order~$6$ are $xy^2$ and $x^2y^2$, and as there are no elements of order~$12$ in~$G$, the possible ramification types for a surface of genus~$3$ are $[2,2,2,3]$. Indeed, no spherical generating system for~$G$ can have type $[3,4,4]$, since $z$ then cannot be generated. Additionally the product of 2 elements of order~$6$ is either $1$, $x$ or $x^2$. Thus the types $[2,6,6]$ and $[3,6,6]$ do not give spherical generating systems. To finish this case, we observe that for $[2,2,2,4]$ and $[2,2,2,3]$ to be generating sets, at least one element of order~2 must be $xz$ or $x^2z$, which are conjugates of each other. Hence (\ref{sigmacondition}) never holds here.

For $G=\text{Dic}_6=\langle x,y\mid x^{12}=1, y^2=x^6,\,yxy=x^{-1}\rangle$, as there is only one element of order~2, which is central, this group cannot occur here.
    
    Finally for $G=D_{12}=\langle x,y\mid x^{12}=y^2=1,\,yxy=x^{-1}\rangle$, since all elements of orders 3, 4, and 6 lie in the unique cyclic subgroup of order~$12$, the possible ramification types for a surface of genus~3 are $[2,2,2,3]$, $[2,4,12]$ and $[2,6,6]$, and for genus~4 the only possibility is $[2,2,2,4]$. As furthermore the cyclic subgroups of orders 4 and 6 intersect non-trivially, for at least one action on a surface of genus~3 the ramification type is $[2,2,2,3]$. Observe that since the types must correspond to a spherical generating set, we have that at least one element of order~2 for each type is of the form $x^ky$ for some $k\ne 0$. As for any word representing a given element, the total exponent of $y$ modulo~2 is constant, we can only have exactly two generators of order~2 being a word of the form $x^\ell y$ for some $\ell$. Hence in both ramification types $[2,2,2,3]$ and $[2,2,2,4]$ there is one generator which is the central element $x^6$ of order~$2$. So (\ref{sigmacondition}) does not hold, and this completes the proof.
\end{proof}

We remark that  Broughton~\cite[Table~5]{Broughton91} showed that there are only four possibilities for~$G$, if $|G|=24$ and if $G$ acts on a surface of genus 3. These are  $S_4, D_{2,12,5}, Z_2\times A_4$, and $\text{SL}_2(3)$, and Broughton also gives the possible options for the ramification types. Hence in the above proof, one could reduce to considering just these four groups and the corresponding allowed ramification types. 

\begin{thm}\label{thm:48}
The order of $G$ is not $48$.
\end{thm}

\begin{proof}
From the proof of Lemma~\ref{lem:2-divisors}, if $|G|=48$ there are three options for the set of $\alpha(A_i)$'s:
\[
\alpha(A_1)=16,\quad \alpha(A_2)=\alpha(A_3)=24,\quad \alpha(A_4)=12,
\]
or
\[
\alpha(A_1)=16,\quad \alpha(A_2)=\alpha(A_3)=\alpha(A_4)=\alpha(A_5)=24,
\]
or
\begin{equation}
    \label{eq:third}
\alpha(A_1)=8,\quad \alpha(A_2)=\alpha(A_3)=\alpha(A_4)=24.
\end{equation}
In the first, there are actions on surfaces of genera 3, 4 and 5. In the second, only on surfaces of genera 3 and 4, and in the third, on surfaces of genera 3 and 7. From \cite[Table~3]{Pol09}, which accounts for a missing case in \cite[Table~5]{Broughton91}, among the groups of order 48, only $Z_4^2\rtimes Z_3$, $Z_2\times S_4$ and the group
\begin{align*}
H&:=\langle x,y,z,w,t\mid x^2=z^2=w^2=t,\,y^3=t^2=[x,y]=[x,z]=1,\\
&\qquad\qquad\qquad\qquad\qquad\qquad\qquad\qquad\qquad yzy^{-1}=w,\,ywy^{-1}=zw,\,zwz^{-1}=wt\rangle
\end{align*}
can act on a surface of genus 3, with corresponding ramification type $[3,3,4]$, $[2,4,6]$ and $[2,3,12]$ respectively. 

Suppose first that $G=Z_2\times S_4$. Note that $G$ has no element of order 8 or order~12. Hence as the only ramification type associated to $\alpha(A)=16$ is $[2,4,8]$, we see that only the third option~\eqref{eq:third} is possible. Now the possible ramification types associated to genus~7 are $[2,2,2,4]$ or $[4,4,4]$. As the squares of elements of order 4 form one conjugacy class, condition~\eqref{sigmacondition} never holds.

We next consider 
\[
G=Z_4^2\rtimes Z_3=\langle x,y,z\mid x^3=y^4=z^4=[y,z]=1,\, xyx^{-1}=z,\,xzx^{-1}=(yz)^{-1}\rangle.
\]
A direct computation shows that there are only two conjugacy classes of elements of order~3, determined by $x$ and $x^{-1}$ respectively, and each are of size 16. The remaining conjugacy classes consist of one class of elements of order 2 determined by $y^2$, and four classes of elements of order 4 determined by $y$, $yz$, $yz^2$ and $y^{-1}z$ respectively. Each of these five conjugacy classes is of size 3. Certainly here $G$ also has no element of order 8 or 12. Hence only option \eqref{eq:third} needs to be considered. For the action on a surface of genus~7, note that the  ramification types $[2,2,2,4]$ and $[4,4,4]$ are not possible, since all elements of order 2 and of order 4 lie in the normal subgroup  $Z_4^2$. Hence there is no action on a surface of genus~7.

So it remains to deal with the case $G=H$. First observe that the following relations follow from the presentation of~$G$:
\[
[x,t]=[z,t]=[w,t]=1.
\]
We further obtain that $[y,t]=1$. Indeed, from $yz^3y^{-1}=w^3$ we obtain $yzty^{-1}=wt$. Substituting $yz=wy$ into the previous equation gives 
$wyty^{-1}=wt$, which then yields $yt=ty$. Lastly, we deduce from the presentation of~$G$ that 
$[x,w]=1$; indeed, from $yz=wy$ we have
\begin{align*}
  (xwy)^2&=(xyz)^2\\
    &=x^2(yz)^2\\
    &=x^2yzy^{-1}y^{-1}z\\
    &=x^2wy^{-1}z.
\end{align*}
As $(xwy)^2=xwyxwy=xwxywy$, we obtain from the above that
\[
wxywy=xwy^{-1}z.
\]
Recalling that $yz=wy$, we then get
\[
wxy^{-1}z =xwy^{-1}z,
\]
and so $wx=xw$, as claimed.

Note also that
\[
zwz^{-1}=w^{-1},
\]
and since $z^2=w^2=t\in Z(G)$, we further have that
\[
zwz^{-1}=z^{-1}wz,\qquad wzw^{-1}=w^{-1}zw, \qquad zyz^{-1}=z^{-1}yz,  \qquad\textup{and}\qquad   wyw^{-1}=w^{-1}yw.
\]
Next, using the relations $yz=wy$ and $yw=zwy$, 
we compute that
\[
z^{-1}yz=z^{-1}wy=z^2yw=yw^{-1},
\]
and from $zw=w^{-1}z$, 
\[
w^{-1}zw=z^{-1}.
\]
Likewise, from $yw=zwy$, we obtain
\[
w^{-1}yw=w^{-1}zwy=z^{-1}y.
\]

\bigskip

For convenience, we list the non-trivial conjugacy classes in~$G$:
\begin{align*}
    \textup{ccl}(t)&=\{t\}\\
     \textup{ccl}(x)&=\{x\}\\
       \textup{ccl}(xt)&=\{xt\}\\
      \textup{ccl}(z)&=\{z,w,zw,wt,zt,zwt\}\text{ with elements of order }4\\
        \textup{ccl}(xz)&=\{xz,xw,xzw,xwt,xzt,xzwt\}\text{ with elements of order }2\\
       \textup{ccl}(y)&=\{y,ywt, zyt=yzwt, wyt=yzt\}\text{ with elements of order }3\\
         \textup{ccl}(y^{-1})&=\{y^{-1},y^{-1}zw, y^{-1}z, y^{-1}w\}\text{ with elements of order }3\\
          \textup{ccl}(yt)&=\{yt,yw, yzw, yz\}\text{ with elements of order }6\\
          \textup{ccl}(y^{-1}t)&=\{y^{-1}t,y^{-1}zwt, y^{-1}zt, y^{-1}wt\}\text{ with elements of order }6\\
         \textup{ccl}(xy)&=\{xy,xywt, xyzwt, xyzt\}\text{ with elements of order }12\\
           \textup{ccl}(xyt)&=\{xyt,xyw, xyzw, xyz\}\text{ with elements of order }12\\  
             \textup{ccl}(xy^{-1})&=\{xy^{-1},xy^{-1}zw, xy^{-1}z, xy^{-1}w\}\text{ with elements of order }12\\
           \textup{ccl}(xy^{-1}t)&=\{xy^{-1}t,xy^{-1}zwt, xy^{-1}zt, xy^{-1}wt\}\text{ with elements of order }12\\ 
\end{align*}
As there is no element of order~8 in~$G$, there is no action on a surface of genus~4. Indeed, the only possible ramification type for genus 4 is $[2,4,8]$. Therefore we are left with only the option~\eqref{eq:third}, with the possible ramification types for a surface of genus 7 being $[2,2,2,4]$ or $[4,4,4]$. Clearly the type $[4,4,4]$ is impossible, as the only elements of order~4 all lie in the union of the conjugacy class of~$z$ with $\{x,xt\}$, and $y\notin\langle \textup{ccl}(z)\cup \{x,xt\}\rangle$. Similarly, from the above data, we see that a system of ramification type $[2,2,2,4]$ can never be a generating set for~$G$. The proof is now complete.
\end{proof}

\begin{proof}
    [Proof of Theorem \ref{thm:main}]
    This follows immediately from Propositions \ref{prop:divisors-2-5} and \ref{prop:prime-power-no}, and Theorems \ref{thm:order-168}, \ref{thm:24} and \ref{thm:48}.
\end{proof}

\end{document}